\documentclass{amsart}
\usepackage{graphicx}
\usepackage{latexsym}
\usepackage{amsfonts}
\usepackage{color}
\usepackage[all]{xy}
\usepackage{amssymb, mathrsfs, amsfonts, amsmath}
\usepackage{amsbsy}
\usepackage{amsfonts}
\newtheorem*{acknowledgement}{Acknowledgement}

\newtheorem{lemma}{Lemma}

\newtheorem{theorem}{Theorem}

\numberwithin{equation}{section}

\begin{document}
\title[Reduction of gradient Ricci soliton equations]{Reduction of gradient Ricci soliton equations}
 \author{Benedito Leandro}
\address{Universidade Federal de Jata\'i, Center of Exact Sciences, BR 364, km 195, 3800, 75801-615, Jata\'i-GO, Brazil.}
 \email{bleandroneto@gmail.com}

 \thanks{The second author was supported by FAPDF 0193.001346/2016}
 \author{Jo\~ao Paulo dos Santos}
\address{ Universidade de Bras\'ilia, Department of Mathematics, 70910-900, Bras\'ilia-DF, Brazil.}
 \email{j.p.santos@mat.unb.br}

\date{\today}

\begin{abstract}

We consider gradient Ricci solitons conformal to a $n$-dimensional pseudo-Euclidean space and we completely describe the most general ansatz that reduces the resulting system of partial differential equations to a system of ordinary differential equations. As a consequence, the gradient Ricci solitons that arise from the reduced system are invariant under the action of either an $(n-1)$-dimensional translation group or the pseudo-orthogonal group acting on the corresponding $n$-dimensional pseudo-Euclidean space. 
\end{abstract}

\maketitle
\section{Introduction and Main Results}

A pseudo-Riemannian manifold $(M^n,g)$ endowed with a smooth vector field $X$ is a Ricci soliton if
\begin{equation}
Ric_g + \dfrac{1}{2} \mathcal{L}_X g  = \lambda g, \label{ricci-soliton}
\end{equation}
where $Ric_g$ is the Ricci tensor, $\mathcal{L}_{X}g$ is the Lie derivative in the direction of $X$ and $\lambda$ is a real constant.  The vector field $X$ is called potential vector field. The Ricci soliton is called shrinking when $\lambda>0$, steady when $\lambda=0$, and expanding when $\lambda<0$.
If $X$ is the gradient of a smooth function $f$, the Ricci soliton equation takes the form
\begin{equation}
 Ric_g + Hess_f = \lambda g, \label{gradient-ricci-soliton}
\end{equation}
where $Hess_f$ is the Hessian of $f$. In this case,  $(M^n,g)$ is called gradient Ricci soliton and the function $f$ is called potential function.

Ricci solitons are natural generalizations of Einstein metrics. They arises as self-similar solutions and they play an important role as singularity models for the Hamilton's Ricci flow. The theory of Ricci solitons has been systematically studied over the years and a good survey on the subject can be found in  \cite{cao2009,cao11,cao22}. In this work, we focus our attention on gradient Ricci solitons conformal to a $n$-dimensional pseudo Euclidean espace. 





For two-dimensional Ricci solitons, Hamilton \cite{hamilton-1} proved that any closed Ricci soliton must have constant curvature. Hamilton also obtained the first complete steady gradient Ricci soliton, known as the \emph{cigar soliton}. This soliton is also known as the \emph{Witten's black hole} \cite{witten}.  A complete classification of two dimensional gradient Ricci solitons was given by Bernstein and Mettler \cite{bernstein}. A similar result is also provided by Ramos \cite{ramos}. 

When the dimension of the soliton is $n \geq 3$, Bryant \cite{bryant} showed that there exists a complete, rotationally symmetric steady gradient Ricci soliton on $\mathbb{R}^n$, which is unique up to homothety (see also Chapter 1, section 4 in \cite{chow}). This soliton is known as the \emph{Bryant soliton}.

The Bryant soliton is a key to the classification of locally conformally flat steady gradient ricci solitons. In fact, Cao and Chen \cite{cao} proved that a complete noncompact locally conformally flat gradient steady Ricci soliton is either flat or isometric to the Bryant soliton. For the general case, 
Fern\'{a}ndez-L\'{o}pez and Garc\'ia-R\'io \cite{FernandoGarcia}
proved that a complete locally conformally flat gradient shrinking or steady Ricci soliton is rotationally symmetric. For the expanding case, they also proved the rotational symmetry when the curvature operator is nonnegative. Consequently, by using the previous results of Kotschwar \cite{kotschwar} and Cao and Chen's \cite{cao}, Fern\'andez-L\'opez and Garc\'ia-R\'io proved that any locally conformally flat, complete, simply connected, gradient Ricci soliton is isometric to $\mathbb{R}\times\mathbb{S}^{n-1}$, $\mathbb{R}^{n}$ or $\mathbb{S}^{n}$, if it is shrinking, or isometric to $\mathbb{R}^n$ or the Bryant soliton, if it is steady.  


Pseudo-Riemannian Ricci solitons have been recently investigated, as we can see in \cite{lor1}, \cite{lor2}, \cite{lor3} and \cite{lor4}, specially in the Lorentzian case. In \cite{lor3}, a local characterization of conformally flat Lorentzian gradient Ricci solitons is given. Explicit pseudo-Riemannian conformally flat gradient steady Ricci solitons was obtained by Barbosa, Pina and Tenenblat in \cite{keti}. The authors firstly reduced the soliton equation to an ODE system by considering a substitution by a function invariant under translations in a pseudo-Euclidean space. In the steady case, they provided all solutions of the reduced system. Consequently, a family of pseudo-Riemannian gradient steady Ricci soliton invariant by translations is given. 


In general, an \emph{ansatz} is a substitution that transforms the original PDE into an ODE or a PDE with less independent variables. The Lie point symmetry groups for PDE provides \emph{ansatze} by considering invariant functions under the action of a symmetry group using the algebra of infinitesimal generators. Another method for finding \emph{ansatz} is known as direct method of reduction and was introduced in a systematic way by Clarkson and Kruskal \cite{clarkson}, where the Boussinesq equation was considered.  Although the direct method has been introduced for two-dimensional equations, it was also successful when multiple variables were considered, as we can see, for example, for the nonlinear multi-dimensional wave equations \cite{irina2}. The method consists in considering new variables given by unknown \emph{ansatze} for a given PDE and then it is required that the equivalent system of equations to be a system of ODE or a system of PDE with less independent variables. With this requirement, one obtains a set of partial differential equations for the \emph{ansatze}. Once the \emph{ansatze} are obtained, one has the reduced system. 


In this paper, we consider gradient Ricci solitons conformal to a $n$-dimensional pseudo-Euclidean space $(\mathbb{R}^n,g)$. By using the direct method of reduction, we generalize the results obtained by Barbosa, Pina and Tenenblat \cite{keti} and we obtain the most general \emph{ansatz} that reduces the correspondent system of PDE to an ODE system (Theorem \ref{thm-most-general}). It turns out that the function which plays the role of \emph{ansatz} must be an invariant function under the action of the semi-orthogonal group or an invariant function under the group of translations. The reduced system of ODE's is given in Theorem \ref{thm-reduced}. A particular case of the reduced system is given in Theorem \ref{thm-special}, where a constraint is considered in order to reduce the order of the system.

In order to state our results, let $(\mathbb{R}^{n}, g)$ be the standard pseudo-Euclidean space with coordinates $(x_{1}, \cdots, x_{n})$ and metric components $g_{ij} = \delta_{ij}\varepsilon_{i}$, $1\leq i, j\leq n$, where 
$\varepsilon_{i} = \pm1$, with at least one $\varepsilon_{i} = 1$. We want to find smooth functions $\varphi$ and $f$ defined on an open subset $\Omega \subset \mathbb{R}^n$ such that, for $\bar{g}$ given by $$\bar{g}=\frac{g}{\varphi^2}$$ 
$(\Omega, \bar{g})$ is a gradient Ricci soliton with potential function $f$, i.e.,
\begin{eqnarray}
Ric_{\bar{g}}+Hess_{\bar{g}}(f)=\lambda\bar{g}, \label{solitonequation-bar}
\end{eqnarray}
where $Ric_{\bar{g}}$ and $Hess_{\bar{g}}(f)$ are, respectively, the Ricci tensor and the Hessian of the metric $\bar{g}$. In what follows, we denote that the directional derivatives of $\varphi$ and $f$ by
\begin{eqnarray*}
\frac{\partial\varphi}{\partial x_{i}}=\varphi_{,i}\quad\mbox{and}\quad\frac{\partial f}{\partial x_{i}}=f_{,i}.
\end{eqnarray*}
In this scenario, Theorem 1.3 in Barbosa, Pina and Tenenblat \cite{keti} supplies  the correspondent PDE system for the equation \eqref{solitonequation-bar}, i.e., $(\Omega, \bar{g})$ is a gradient Ricci soliton with potential function $f$ if, and only if, 

\begin{equation}\label{pde-soliton-ij}
(n-2)\varphi_{,ij}+\varphi{f}_{,ij}+\varphi_{,i}f_{,j}+\varphi_{,j}f_{,i}=0,\quad i\neq j
\end{equation}
and for each $i$
\begin{eqnarray}\label{pde-soliton-ii}
&&\varphi[(n-2)\varphi_{,ii}+\varphi{f}_{,ii}+2\varphi_{,i}f_{,i}]\nonumber\\
&&+\varepsilon_{i}\displaystyle\sum_{k}\varepsilon_{k}\left[\varphi\varphi_{,kk}-(n-1)\varphi^{2}_{,k}-\varphi\varphi_{,k}f_{,k}\right]=\varepsilon_{i}\lambda.
\end{eqnarray}
It is important to note that, although Theorem 1.3 in \cite{keti} requires $n\geq 3$, it is not necessary to put any restriction on the dimension. 

Our main goal is to find a smooth function $\xi: \Omega \subset \mathbb{R}^n \rightarrow \mathbb{R}$ that will be an \emph{ansatz} such that equations \eqref{pde-soliton-ij} and \eqref{pde-soliton-ii} are reduced to ordinary differential equations. Specifically, we want to find $\xi : \Omega \subset \mathbb{R}^n \rightarrow \mathbb{R}$ such that $f \circ \xi$ and $\varphi \circ \xi$ are solutions for such equations.

The next theorem prove that the most general form for the function $\xi$. 

\begin{theorem}\label{thm-most-general}
Let $(\mathbb{R}^{n}, g)$ be the standard pseudo-Euclidean space with coordinates $(x_{1}, \cdots, x_{n})$ and metric components $g_{ij} = \delta_{ij}\varepsilon_{i}$, $1\leq i, j\leq n$, where 
$\varepsilon_{i} = \pm1$, with at least one $\varepsilon_{i} = 1$. Then, there exists  a smooth real function $\xi : \Omega \subset \mathbb{R}^n \rightarrow \mathbb{R}$ such the PDE system given by the equations \eqref{pde-soliton-ij} and \eqref{pde-soliton-ii} are reduced to an ODE system with independent variable $\xi$ if, and only if, 
\begin{equation}
\xi(x_1, \ldots, x_n) = \Psi \left(\displaystyle\sum_{k} (\tau\varepsilon_{k}x_{k}^{2}+ \alpha_{k}x_{k}+\beta_{k}) \right), \label{eq-most-general}
\end{equation}
where $\tau, \, \alpha_{k}, \, \beta_{k} \in \mathbb{R}$ and $\Psi$ is a smooth real function.
\end{theorem}

Firstly, let us observe that Theorem \ref{thm-most-general} is sharp in the following sense: there is no other function $\xi$, that depends on all variables $(x_1, \ldots, x_n)$, such that equations \eqref{pde-soliton-ij} and \eqref{pde-soliton-ii} reduce to an ODE system. Let us also observe that the level functions of such a $\xi$ provides a foliation of the correspondent gradient Ricci soliton by hypersurfaces invariant under the action of the pseudo-orthogonal group in when $\tau \neq 0$ (up to change of coordinates) or the group of translations, when $\tau = 0$, which is precisely the case considered in \cite{keti}. Fern\'andez-L\'opez  and Garc\'ia-R\'io \cite{garciario}, using the local decomposition of a Ricci soliton metric into a warped product metric (see also \cite{catino}) proved that a locally conformally flat gradient Ricci soliton is rotationally symmetric. However, for a expanding soliton they require that the curvature operator must be nonnegative (cf. Remark 1). From Theorem \ref{thm-most-general}, we prove without any assumption on the curvature that a pseudo-Riemannian conformally flat gradient expanding (shrinking or steady) Ricci soliton, foliated by $n-1$ dimensional subsets invariant under isometries, is pseudo-rotationally symmetric (up to change of variables). Moreover, our proof is quite different in the sense that we do not use the warped product structure.

Once we have the most general \emph{ansatz}, one has the reduced system of ODE's. Since $\xi$ given in \eqref{eq-most-general} is a function on the polynomial $\displaystyle \sum_k (\tau\varepsilon_{k}x_{k}^{2}+ \alpha_{k}x_{k}+\beta_{k})$,  we will consider $\varphi$ and $f$ smooth real functions depending on the variable $\xi = \sum_k (\tau\varepsilon_{k}x_{k}^{2}+ \alpha_{k}x_{k}+\beta_{k}) $. In this case, let us write
$$
\varphi':= \dfrac{d \varphi}{d \xi}\,\, \textnormal{ and } \,\, f':=\dfrac{df}{d\xi}.
$$
The reduced system is given in the following theorem:

\begin{theorem}\label{thm-reduced}
Let $(\mathbb{R}^{n}, g)$ be the standard pseudo-Euclidean space with coordinates $(x_{1}, \cdots, x_{n})$ and metric components $g_{ij} = \delta_{ij}\varepsilon_{i}$, $1\leq i, j\leq n$, where $\varepsilon_{i} = \pm1$, with at least one $\varepsilon_{i} = 1$. Let $\Omega \subset \mathbb{R}^{n}$ be open subset and consider smooth functions $\varphi(\xi)$ and $f(\xi)$ such that $\xi : \Omega \subset \mathbb{R}^n \rightarrow \mathbb{R}$ is given by $\xi = \sum_k (\tau\varepsilon_{k}x_{k}^{2}+ \alpha_{k}x_{k}+\beta_{k})$. Then $(\Omega, \bar{g})$ is gradient Ricci soliton with potential function $f(\xi)$ and $\bar{g} = \frac{g}{(\varphi(\xi))^{2}}$ if and only if the functions $\varphi$ and $f$ satisfy 
\begin{eqnarray}\label{eq-reduced-1}
(n-2)\varphi''+f''\varphi+2\varphi'f'=0
\end{eqnarray}
and
\begin{eqnarray}\label{eq-reduced-2}
2\tau\varphi[2(n-1)\varphi'+\varphi f']
+\left[\varphi\varphi''-(n-1)(\varphi')^{2}-\varphi\varphi'f'\right](4\tau\xi+ \Lambda)=\lambda,
\end{eqnarray}
 where $\tau, \, \alpha_{k}, \, \beta_{k} \in \mathbb{R}$ and $\Lambda=\displaystyle\sum_{k} (\varepsilon_k \alpha^{2}_{k} -4\tau \beta_{k})$.
\end{theorem}

An example of complete solution for Theorem \ref{thm-reduced} is the Gaussian soliton (cf. \cite{cao2009,chow}), shrinking and expanding (cf. \cite{cao2009}). In fact, it is enough to consider $\varphi=k$, where $k$ is constant, and $\tau \neq 0$. Therefore, from \eqref{eq-reduced-1} we have that $f(\xi)=a_{1}\xi+a_{2},$ where $a_{1},a_{2}\in\mathbb{R}.$ From \eqref{eq-reduced-2} we can see that $a_{1}=\frac{\lambda}{2\tau k^{2}}$. Considering $\alpha_{i}=\beta_{i}=0$ for all $i\in\{1,\ldots, n\}$ in the expression of $\xi$, we have the Gaussian soliton.

On the other hand, when $f$ is constant, it follows from \eqref{eq-reduced-1} that $\varphi(\xi) = b_1 \xi + b_2$, which implies $\overline{g}$ has constant curvature. Consequently, the space forms are complete solutions of Theorem \ref{thm-reduced}.

When $\tau = 0$, we recover the ODE's given in \cite{keti}. In this case, if we consider $\lambda=0$ we have the explicit solutions given in \cite{keti}.

Now, we consider a special case of Theorem \ref{thm-reduced} when $\tau \neq 0$. By considering a constraint that relates the second derivatives of $\varphi$ and $f$, we will get a system of ordinary differential equations of first order:

\begin{theorem}\label{thm-special}
Let $(\mathbb{R}^{n}, g)$ be the standard pseudo-Euclidean space with coordinates $(x_{1}, \cdots, x_{n})$ and metric components $g_{ij} = \delta_{ij}\varepsilon_{i}$, $1\leq i, j\leq n$, where $\varepsilon_{i} = \pm1$, with at least one $\varepsilon_{i} = 1$. Let $\Omega \subset \mathbb{R}^{n}$ be open subset and consider smooth functions $\varphi(\xi)$ and $f(\xi)$ such that $\xi : \Omega \subset \mathbb{R}^n \rightarrow \mathbb{R}$ is given by $\xi = \sum_k (\tau\varepsilon_{k}x_{k}^{2}+ \alpha_{k}x_{k}+\beta_{k})$. Then $(\Omega, \bar{g})$ is gradient Ricci soliton with potential function $f(\xi)$, $\bar{g} = \frac{g}{(\varphi(\xi))^{2}}$ such that $2n\varphi''+ \varphi f''=0$ if, and only if the functions $\varphi$ and $f$ satisfy 

\begin{equation}
\varphi(\xi)^2 = h(\xi), \label{special-phi}
\end{equation}
\begin{equation}\label{special-f}
f'(\xi)=c_1 h(\xi)^{-2n/(n+2)},
\end{equation}
for
\begin{equation}
(n-1)h' + c_1 h^{-(n-2)/(n+2)} = c_2 (4 \tau \xi + \Lambda) + \dfrac{\lambda}{2 \tau} \label{thm-special-F}
\end{equation}
where $\tau, \, \alpha_{k}, \, \beta_{k}, c_1, \, c_2 \in \mathbb{R}$ and $\Lambda=\displaystyle\sum_{k} (\varepsilon_k \alpha^{2}_{k} -4\tau \beta_{k})$.
\end{theorem}

When we consider $n=2$ in Theorem \ref{thm-special}, we have
\begin{equation*}
h(\xi) = 2 c_2 \tau \xi^2 + \left( c_2 \Lambda + \dfrac{\lambda}{2 \tau} - c_1 \right) \xi + c_3.
\end{equation*}
In the Riemannian case, a particular choice of constants leads us to the cigar soliton. In fact, for $c_2=\lambda=0$ and $c_1=-c_3=-1$, whe have $h(\xi) = \xi+1$. By considering $\alpha_i = \beta_i = 0$ we have
\begin{eqnarray*}
	\bar{g}=\frac{dx_{1}^{2}+dx_{2}^{2}}{1+x_{1}^{2}+x_{2}^{2}}
\end{eqnarray*}
and
\begin{eqnarray*}
f(x_{1}, x_{2})=-\ln(1+x_{1}^{2}+x_{2}^{2}).
\end{eqnarray*}

\

\section{Proof of the Main Results}

This section is reserved to the demonstration of the main results of this manuscript. Let us start with the following lemma:

\begin{lemma}\label{lemma-contraction}
Let $(\mathbb{R}^{n}, g)$ be a pseudo-Euclidean space with cartesian coordinates $(x_{1},\ldots , x_{n})$ and metric components $g_{ij} = \delta_{ij}\varepsilon_{i}$, $1\leq i,j\leq n$, where $\varepsilon_{i}=\pm1$ which at least one equal to $1$. Let $\Omega \subset\mathbb{R}^{n}$ be an open subset and $f,\varphi : \Omega\rightarrow\mathbb{R}$ be smooth functions. If $(\Omega, \bar{g})$ is a gradient Ricci soliton equation with potential function $f$ and $\bar{g}=\dfrac{g}{\varphi^2}$, then the functions $\varphi$ and $f$ satisfy 
\begin{eqnarray}\label{eq33-1}
\displaystyle\sum_{k}\varepsilon_{k}\left[2(n-1)\varphi\varphi_{,kk}-n(n-1)\varphi_{,k}^{2}+\varphi^{2}f_{,kk}-(n-2)\varphi\varphi_{,k}f_{,k}\right]=n\lambda.
\end{eqnarray}
\end{lemma}

\begin{proof}
When we contract Equation \eqref{solitonequation-bar}, it is a straightforward computation that
\begin{eqnarray}\label{solitoneq2}
R_{\bar{g}}+\Delta_{\bar{g}}f=n\lambda.
\end{eqnarray}

It is well known that if $\bar{g}=\frac{1}{\varphi^{2}}g$ (cf. Lemma 1 in \cite{khunel}), then
$$Ric_{\bar{g}}=\frac{1}{\varphi^{2}}\{(n-2)\varphi Hess_{g}(\varphi) + [\varphi\Delta_{g}\varphi - (n-1)|\nabla_{g}\varphi|^{2}]g\}.$$
Hence, the scalar curvature of $\bar{g}$ is given by
\begin{eqnarray}\label{scalar}
R_{\bar{g}}&=&\displaystyle\sum_{k=1}^{n}\varepsilon_{k}\varphi^{2}\left(Ric_{\bar{g}}\right)_{kk}=(n-1)(2\varphi\Delta_{g}\varphi - n|\nabla_{g}\varphi|^{2})\nonumber\\
&=&(n-1)\left[2\varphi\displaystyle\sum_{k}\varepsilon_{k}\varphi_{,kk}-n\displaystyle\sum_{k}\varepsilon_{k}(\varphi_{,k})^{2}\right]\nonumber\\
&=&\displaystyle\sum_{k}\varepsilon_{k}\left[2(n-1)\varphi\varphi_{,kk}-n(n-1)(\varphi_{,k})^{2}\right].
\end{eqnarray}

Remember that
$$Hess_{\bar{g}}(f)_{ij}=f_{,ij}-\displaystyle\sum_{k}\bar{\Gamma}^{k}_{ij}f_{,k},$$
where $\bar{\Gamma}^{k}_{ij}$ are the Christoffel symbols of the metric $\bar{g}$. For $i, j, k$ distinct, we have (cf. \cite{carmo} pg. 161)
\begin{eqnarray}\label{christoffel}
\bar{\Gamma}^{k}_{ij}=0,\quad\bar{\Gamma}^{i}_{ij}=-\frac{\varphi_{,j}}{\varphi},
\quad\bar{\Gamma}^{k}_{ii}=\varepsilon_{i}\varepsilon_{k}\frac{\varphi_{,k}}{\varphi},\quad\bar{\Gamma}^{i}_{ii}=-\frac{\varphi_{,i}}{\varphi}.
\end{eqnarray}
Therefore,
since
\begin{eqnarray*}
Hess_{\bar{g}}(f)_{ii}=f_{,ii}+2\frac{\varphi_{,i}f_{,i}}{\varphi}-\varepsilon_{i}\displaystyle\sum_{k}\varepsilon_{k}\frac{\varphi_{,k}f_{,k}}{\varphi}
\end{eqnarray*}
we obtain that 
\begin{eqnarray}\label{laplaciano}
\Delta_{\bar{g}}f=\varphi^{2}\displaystyle\sum_{i}\varepsilon_{i}Hess_{\bar{g}}(f)_{ii}=\displaystyle\sum_{k}\varepsilon_{k}\left[\varphi^{2}f_{,kk}-(n-2)\varphi\varphi_{,k}f_{,k}\right].
\end{eqnarray}
Then, from \eqref{solitoneq2}, (\ref{scalar}) and (\ref{laplaciano}) we have the result.

\end{proof}

\noindent {\bf Proof of Theorem \ref{thm-most-general}:}
By hypothesis $f(\xi)$ and $\varphi(\xi)$ are functions of $\xi$ then we have
\begin{eqnarray}\label{t1}
&&\varphi_{,i}=\varphi'\xi_{,i},\quad\varphi_{,ij}=\varphi''\xi_{,i}\xi_{,j}+\varphi'\xi_{,ij},\nonumber\\
&&{f}_{,i}={f}'\xi_{,i}\quad\mbox{and}\quad{f}_{,ij}={f}''\xi_{,i}\xi_{,j}+{f}'\xi_{,ij}.
\end{eqnarray}
Then, from (\ref{pde-soliton-ij}) 
\begin{eqnarray*}
(n-2)[\varphi''\xi_{,i}\xi_{,j}+\varphi'\xi_{,ij}]+\varphi[{f}''\xi_{,i}\xi_{,j}+{f}'\xi_{,ij}]+2\varphi'f'\xi_{,i}\xi_{,j}=0.
\end{eqnarray*}
Which implies that,
\begin{eqnarray}\label{s1}
[(n-2)\varphi''+\varphi{f}''+2\varphi'f']\xi_{,i}\xi_{,j}+[(n-2)\varphi'+\varphi{f}']\xi_{,ij}=0.
\end{eqnarray}
Dividing the above equation for $\xi_{,i}\xi_{,j}$ we can infer that
\begin{eqnarray}\label{eq-for-F}
F(\xi)=\frac{\xi_{,ij}}{\xi_{,i}\xi_{,j}},
\end{eqnarray}
for some smooth function $F(\xi)$.

Combining equations \eqref{pde-soliton-ii} and \eqref{eq33-1} we have the following
\begin{eqnarray}\label{M1}
&&n\varepsilon_{i}[(n-2)\varphi_{,ii}+\varphi{f}_{,ii}+2\varphi_{i}f_{i}]\nonumber\\
&&=\displaystyle\sum_{k}\varepsilon_{k}\left[(n-2)\varphi_{,kk}+\varphi{f}_{,kk}+2\varphi_{k}f_{k}\right].
\end{eqnarray}
Since $\varphi_{,ii}=\varphi''\xi_{,i}^{2}+\varphi'\xi_{,ii}$, the same for $f$, we obtain that
\begin{eqnarray}\label{s2}
&&[(n-2)\varphi''+\varphi{f}''+2\varphi'f']\left[n(\varepsilon_{i}\xi_{,i}^{2})-\displaystyle\sum_{k}(\varepsilon_{k}\xi_{,k}^{2})\right]\nonumber\\
&&+[(n-2)\varphi'+\varphi{f}']\left[n(\varepsilon_{i}\xi_{,ii})-\displaystyle\sum_{k}(\varepsilon_{k}\xi_{,kk})\right]=0.
\end{eqnarray}
From \eqref{s1}, \eqref{eq-for-F} and \eqref{s2} we have
\begin{eqnarray}\label{imp1}
&&[(n-2)\varphi'+\varphi{f}']\Bigg\{\left[n(\varepsilon_{i}\xi_{,ii})-\displaystyle\sum_{k}(\varepsilon_{k}\xi_{,kk})\right]\nonumber\\
&&-F(\xi)\left[n(\varepsilon_{i}\xi_{,i}^{2})-\displaystyle\sum_{k}(\varepsilon_{k}\xi_{,k}^{2})\right]\Bigg\}=0.
\end{eqnarray}

Observe that the soliton is trivial if
\begin{eqnarray}\label{ts1}
(n-2)\varphi'+\varphi{f}'=0.
\end{eqnarray}
In fact, supposing that, the first derivative of this equation leads us to $$(n-2)\varphi''+\varphi{f}''+\varphi'f'=0.$$
Therefore, from \eqref{s1} we have $\varphi'f'\xi_{,i}\xi_{,j}=0.$ This equation implies that either $\varphi$ or $f$ are constant functions. In any case, from \eqref{ts1}, we have that $\varphi$ and $f$ must by trivial.

Consequently,  from \eqref{imp1} we conclude that
\begin{eqnarray*}
\left[n(\varepsilon_{i}\xi_{,ii})-\displaystyle\sum_{k}(\varepsilon_{k}\xi_{,kk})\right]=F(\xi)\left[n(\varepsilon_{i}\xi_{,i}^{2})-\displaystyle\sum_{k}(\varepsilon_{k}\xi_{,k}^{2})\right].
\end{eqnarray*}
We rewrite this equation in the  form
\begin{eqnarray*}
n\varepsilon_{i}\left[\xi_{,ii}-F(\xi)\xi_{,i}^{2}\right]=\displaystyle\sum_{k}\varepsilon_{k}\left[\xi_{,kk}-F(\xi)\xi_{,k}^{2}\right],
\end{eqnarray*}
in order to get
\begin{eqnarray}
n\varepsilon_{i}\left[(\xi_{,i})e^{-\int{F}d\xi}\right]_{,i}=\displaystyle\sum_{k}\varepsilon_{k}\left[(\xi_{,k})e^{-\int{F}d\xi}\right]_{,k}. \label{consequence-combining}
\end{eqnarray}
At this point, let us observe that equation (\ref{eq-for-F}) can be integrated as the following
$$
\ln(\xi_{,i}) = \int F d \xi + F_i(\hat{x}_j),
$$
where the symbol $\hat{x}_j$ denotes that $F_{i}$ does not depend on the variable $x_j$. By considering this process for all $j \neq i$, we have
\begin{equation}
\xi_{,i} = e^{\int F d \xi + F_i(x_i)}=G_i G, \label{xi-derivative}
\end{equation}
where $G_i(x_i) = e^{F_i(x_i)}$ and $G(\xi) = e^{\int F d \xi} $. Therefore, equations \eqref{consequence-combining} and \eqref{xi-derivative} imply that


\begin{eqnarray*}
n\varepsilon_{i}G'_{i}=\displaystyle\sum_{k}\varepsilon_{k}G'_{k}.
\end{eqnarray*}
Observe that the left side of the above equation depends only depend on the variable $x_{i}$, thus this implies that
\begin{eqnarray}\label{sera}
\varepsilon_{i}G'_{i}=\varepsilon_{j}G'_{j},\quad\forall i\neq j.
\end{eqnarray}
Since $G_{i}$ is a function which depends only of $x_i$, the only possibility for \eqref{sera} is $G_i(x_{i}) =2\tau\varepsilon_{i}x_{i}+ \alpha_i$, where $\alpha_i$ and $\tau$ are real constants. Since $\frac{\xi_{,i}}{G_i}= G$ we have for all $i \neq j$ that 
\begin{equation} \label{derivatives-xi}
\dfrac{\xi_{,i}}{2\tau\varepsilon_{i}x_{i}+ \alpha_{i}} = \dfrac{\xi_{,j}}{2\tau\varepsilon_{j}x_{j}+ \alpha_{j}}.
\end{equation}
The characteristic for the equation above implies that
\begin{eqnarray}\label{basic_invariant}
\xi = \Psi\left(\displaystyle\sum_{i}\tau\varepsilon_{i}x_{i}^{2}+ \alpha_{i}x_{i}+\beta_{i}\right),
\end{eqnarray}
where $\Psi$ is a smooth function. 


Conversely, let us suppose that $\xi = \Psi\left(\displaystyle\sum_{i}\tau\varepsilon_{i}x_{i}^{2}+ \alpha_{i}x_{i}+\beta_{i}\right)$, for a real smooth function $\Psi$ and constants $\tau, \, \alpha_i, \, \beta_i$. Since we are considering smooth functions by $f(\xi)$ and $\varphi(\xi)$ as solutions of (\ref{pde-soliton-ij}) and (\ref{pde-soliton-ii}), we can suppose, without loss of generality, that $\xi = \displaystyle\sum_{k} U_k (x_k)$, where $U_k(x_k) = \tau \varepsilon_k x_k^2+\alpha_k x_k + \beta_k$, then we have
\begin{eqnarray*}
&&\varphi_{,i}=\varphi'U'_{i},\quad\varphi_{,ij}=\varphi''U'_{i}U'_{j},\quad\varphi_{,ii}=\varphi''(U'_{i})^{2}+\varphi'U''_{i},\nonumber\\
&&{f}_{,i}={f}'U'_{i},\quad{f}_{,ij}={f}''U'_{i}U'_{j}\quad\mbox{and}\quad f_{,ii}=f''(U'_{i})^{2}+f'U''_{i}.
\end{eqnarray*}
Hence, from equations \eqref{pde-soliton-ij} and \eqref{pde-soliton-ii} we get 
\begin{eqnarray}\label{edo1}
[(n-2)\varphi''+f''\varphi+2\varphi'f']U^{\prime}_{i}U^{\prime}_{j}=0
\end{eqnarray}
and for $i=j$
\begin{eqnarray}\label{edo2}
&&\varphi[(n-2)\varphi''+f''\varphi+2\varphi'f'](U^{\prime}_{i})^{2}+\varphi[(n-2)\varphi'+\varphi f']U_{i}''\nonumber\\
&&+\varepsilon_{i}\left\{\left[\varphi\varphi''-(n-1)(\varphi')^{2}-\varphi\varphi'f'\right]\displaystyle\sum_{k}\varepsilon_{k}(U'_{k})^{2}+\varphi\varphi'\displaystyle\sum_{k}\varepsilon_{k}U''_{k}\right\}=\varepsilon_{i}\lambda
\end{eqnarray}
From \eqref{edo1} we get
\begin{eqnarray}\label{edo-p1}
(n-2)\varphi''+f''\varphi+2\varphi'f'=0.
\end{eqnarray}
For the second equation we first have
\begin{equation*}
\begin{array}{rcl}
\displaystyle\sum_{k}\varepsilon_{k}(U^{\prime}_{k})^{2}&=&4\tau\xi+ \Lambda, \\ 
\displaystyle\sum_{k}\varepsilon_{k}U_{k}^{\prime\prime}&=&2n\tau,
\end{array}
\end{equation*}
with $\Lambda=\displaystyle\sum_{k} (\varepsilon_k \alpha_k -4\tau \beta_{k})$. From the \eqref{edo-p1}, equation \eqref{edo2} is rewritten as 
\begin{eqnarray*}
&&2\tau\varepsilon_{i}\varphi[(n-2)\varphi'+\varphi f']\nonumber\\
&&+\varepsilon_{i}\left\{\left[\varphi\varphi''-(n-1)(\varphi')^{2}-\varphi\varphi'f'\right](4\tau\xi+ \Lambda)+2n\tau\varphi\varphi'\right\}=\varepsilon_{i}\lambda.
\end{eqnarray*}
Then we have
\begin{eqnarray}\label{edo-p2}
2\tau\varphi[2(n-1)\varphi'+\varphi f']
+\left[\varphi\varphi''-(n-1)(\varphi')^{2}-\varphi\varphi'f'\right](4\tau\xi+ \Lambda)=\lambda.
\end{eqnarray}

\hfill $\Box$
\

\noindent {\bf Proof of Theorem \ref{thm-reduced}:} Theorem \ref{thm-reduced} is actually a corollary from Theorem \ref{thm-most-general}. Since $\xi=\displaystyle\sum_{i}(\tau\varepsilon_{i}x_{i}^{2}+ \alpha_{i}x_{i}+\beta_{i})$ reduces the system, the ODE system is given by equations (\ref{edo-p1}) and (\ref{edo-p2}).

\hfill $\Box$
\

\noindent {\bf Proof of Theorem \ref{thm-special}:} Firstly, the constraint $ 2n \varphi'' + \varphi f''=0$ applied to equation (\ref{eq-reduced-1}) implies that 
\begin{eqnarray*}
	\frac{(n+2)}{2n}\frac{f''}{f'}=-2\frac{\varphi'}{\varphi}.
\end{eqnarray*}

An integration from the above equation give us 
\begin{eqnarray}\label{exjp3}
	f'=c_{1}\varphi^{-4n/(n+2)};\quad\, c_1 \in \mathbb{R}, \label{dem-special-f}
\end{eqnarray}

Now let us consider the following substitution on equation \eqref{eq-reduced-2}:
	\begin{eqnarray*}
		A&=&\varphi[2(n-1)\varphi'+\varphi\,f'],\nonumber\\
		B&=&2[\varphi\varphi''-(n-1)(\varphi')^{2}-\varphi\varphi'f'], \\
		T&=&4\tau\xi+\Lambda
	\end{eqnarray*}
Then,  \eqref{eq-reduced-2} is rewritten as
    \begin{eqnarray}
    	T'A+TB=2\lambda. \label{reduced-eq-reduced-2}
	\end{eqnarray}
Since $A'=\varphi[2n\varphi''+\varphi\,f'']-B$, the constraint $2n\varphi''+\varphi\,f''=0$ implies that $A'=-B$. Therefore, it follows from \eqref{reduced-eq-reduced-2} that 
    \begin{eqnarray}\label{exjp1}
    	A=c_{2}T+\frac{\lambda}{2\tau};\quad\,c_{1}\in\mathbb{R}.
    \end{eqnarray}	

From \eqref{dem-special-f} and \eqref{exjp1} we get
\begin{eqnarray}\label{exjp2}
	2(n-1)\varphi\varphi'+c_{2}\varphi^{-2(n-2)/(n+2)}=c_{1}T+\frac{\lambda}{2\tau} \label{dem-F}.
\end{eqnarray}
Consequently, if $h(\xi) = \varphi(\xi)^2$, from equations (\ref{dem-special-f}) and (\ref{dem-F}) the proof is done.

\hfill $\Box$

\

\begin{acknowledgement}
The authors are grateful to P. Bonfim, R. Pina and K. Tenenblat for the valuable suggestions on the elaboration of this manuscript. 
\end{acknowledgement}


\begin{thebibliography}{BB}


\bibitem{keti} Barbosa, E., Pina, R., and Tenenblat, K.:  {\em On Gradient Ricci Solitons conformal to a pseudo-Euclidean space.} Israel Journal of Mathematics 200 (2014), 213-224.

\bibitem{lor1} Batat, W.; Brozos-V\'azquez, M.; Garc\'ia-R\'io, E.; Gavino-Fern\'andez, S.:{\em Ricci solitons on Lorentzian manifolds with large isometry groups}. Bull. Lond. Math. Soc. 43 (2011), no. 6, 1219--1227.

\bibitem{bernstein} Bernstein, J., and Mettler, T.
:{\em Two-dimensional gradient Ricci solitons revisited.} International Mathematics Research Notices 2015.1 (2013): 78-98.

\bibitem{lor2} Brozos-V\'azquez, M.; Calvaruso, G.; Garc\'ia-R\'io, E.; Gavino-Fern\'andez, S. :{\em Three-dimensional Lorentzian homogeneous Ricci solitons}. Israel J. Math. 188 (2012), 385--403.


\bibitem{lor3} Brozos-V\'azquez, M.; Garc\'ia-R\'io, E.; Gavino-Fern\'andez, S.: {\em Locally conformally flat Lorentzian gradient Ricci solitons}. J. Geom. Anal. 23 (2013), no. 3, 1196--1212.

\bibitem{bryant} Bryant, R., {\em Local existence of gradient Ricci solitons}, unpublished.

\bibitem{cao2009} Cao, H-D.: {\em Recent progress on Ricci solitons.} Adv. Lect. Math. (ALM), 11 (2009), 1-38.

\bibitem{cao11} Cao, H-D.: {\em Geometry of Ricci solitons.} Chinese Annals of Mathematics, Series B 27.2 (2006): 121-142.

\bibitem{cao22} Cao, H-D., and Chow, B.: {\em Recent developments on the Ricci flow.} Bulletin of the American Mathematical Society 36.1 (1999): 59-74.

\bibitem{cao} Cao, H-D., and Chen, Q.: {\em On locally conformally flat gradient steady Ricci solitons.} Transactions of the American Mathematical Society 364.5 (2012): 2377-2391.



\bibitem{catino} Catino, G., Matengazza, C., Mazzieri, L., and Rimoldi, M. : {\em Locally conformally flat quasi-Einstein manifolds.} Journal fur die reine und angewandte Mathematik (Crelles Journal) 2013.675 (2013): 181-189.

\bibitem{chow} Chow, B., and Knopf, D.: {\em The Ricci flow: an introduction.} Vol. 1. American Mathematical Soc., 2004.

\bibitem{chow1} Chow, B., et al.: {\em The Ricci flow: techniques and applications.} American Mathematical Society, 2007.

\bibitem{clarkson} Clarkson, P. A., and Kruskal, M. D.: {\it New similarity reductions of the Boussinesq equation.} J. Math. Phys. 30, no. 10, 2201-2213 (1989).

\bibitem{carmo} do Carmo, M. P.: {\em Riemannian geometry.} Birkhauser, 1992.



\bibitem{FernandoGarcia} M. Fern\'andez-L\'opez, M., and Garc\'ia-R\'io, E.: {\em Rigidity of shrinking Ricci solitons.} Math. Z., 269 (2011), 461-466.

\bibitem{irina2} Fushchich, W. I.,  Zhdanov, R. Z., and Yegorchenko, I. A.: {\it On the reduction of the nonlinear multi-dimensional wave equations and compatibility of the D'Alembert-Hamilton system.} J. Math. Anal. Appl. 161, no. 2, 352-360 (1991).


\bibitem{hamilton-1} R. S. Hamilton, \emph{The Ricci flow on surfaces}. Mathematics and general relativity (Santa Cruz, CA, 1986), 237--262, Contemp. Math., 71, Amer. Math. Soc., Providence, RI, 1988.

\bibitem{garciario} Fern\'andez-L\'opez, M., and Garc\'ia-R\'io, E.: {\em A note on locally conformally flat gradient Ricci solitons.} Geometriae Dedicata 168.1 (2014): 1-7.


\bibitem{kotschwar} Kotschwar, B.: {\em On rotationally invariant shrinking Ricci solitons.} Pacific Journal of Mathematics 236.1 (2008): 73-88.

\bibitem{khunel}Khunel, W.: {\em Conformal transformations between Einstein spaces.} In: Conformal Geometry. ed. by R.S.Kulkarni and U.Pinkall , aspects of math. E, vol. 12 Vieweg, Braunschweig (1988), 105-146.



\bibitem{lor4} Onda, Kensuke. Lorentz Ricci solitons on 3-dimensional Lie groups. Geom. Dedicata 147 (2010), 313--322.


\bibitem{ramos} Ramos, D.: {\em Gradient Ricci solitons on surfaces.} (2013): preprint arXiv:1304.6391 [math.DG].





\bibitem{witten} Witten, Ed.: {\em String theory and black holes}. Phys. Rev. D (3) 44 (1991), no. 2, 314--324.


\end{thebibliography}
\end{document}